\documentclass[12pt]{amsart}

\usepackage{amsmath}
\usepackage{amssymb}
\usepackage{amsfonts}
\usepackage{amsthm}
\usepackage{enumerate}
\usepackage{hyperref}
\usepackage{graphicx, xcolor}
\usepackage{color}

\theoremstyle{plain}
\newtheorem{thm}{Theorem}[section]

\newtheorem{lem}[thm]{Lemma}
\newtheorem{cor}[thm]{Corollary}

\theoremstyle{definition}

\newtheorem{dfns-rems}[thm]{Definitions and Remarks}
\newtheorem{notas-rems}[thm]{Notations and Remarks}
\newtheorem{exmps-rems}[thm]{Examples and Remarks}
\newtheorem{obs}[thm]{Observation}

\begin{document}


\title[Signed graphs cospectral with the path]{Signed graphs cospectral with the path}


\author[S. Akbari, W.H. Haemers, H.R. Maimani and L. Parsaei Majd]{Saieed Akbari}

\address{S. Akbari, Department of Mathematical Sciences, Sharif University of Technology, Tehran, Iran.}
\email{s$_{-}$akbari@sharif.ir}

\author[]{Willem H. Haemers}
\address{W.H. Haemers, Department of Econometrics and Operations Research,
Tilburg University, Tilburg, The Netherlands.}
\email{haemers@uvt.nl}

\author[]{Hamid Reza Maimani}
\address{H.R. Maimani, Mathematics Section, Department of Basic Sciences,
Shahid Rajaee Teacher Training University, P.O. Box 16785-163, Tehran,
Iran.}
\email{maimani@ipm.ir}

\author[]{Leila Parsaei Majd}
\address{L. Parsaei Majd, Mathematics Section, Department of Basic Sciences,
Shahid Rajaee Teacher Training University, P.O. Box 16785-163, Tehran,
Iran, and School of Mathematics, Institute for Research in Fundamental
Sciences (IPM), P.O. Box 19395-5746, Tehran, Iran.}
\email{leila.parsaei84@yahoo.com}


\begin{abstract}
A signed graph $\Gamma$ is said to be determined by its spectrum if every signed graph with the same spectrum
as $\Gamma$ is switching isomorphic with $\Gamma$.
Here it is proved that the path $P_n$, interpreted as a signed graph, is determined by its spectrum
if and only if $n\equiv 0, 1$, or 2 (mod 4), unless $n\in\{8, 13, 14, 17, 29\}$, or $n=3$.
\\
AMS subject classification 05C50, 05C22.
\end{abstract}

\maketitle


\section{Introduction} \label{sec1}

Throughout this paper all graphs are simple, without loops or parallel edges.
A \textit{signed graph} $\Gamma=(G, \sigma)$ (with $G=(V, E)$) is a graph with the vertex set $V$ and the edge set $E$ together
with a function $\sigma : E \rightarrow \{-1, +1\}$, called the \textit{signature function}.
So, every edge becomes either positive or negative.
The adjacency matrix $A$ of $\Gamma$ is obtained from the adjacency matrix of the underlying graph $G$,
by replacing $1$ by $-1$ whenever the corresponding edge is negative.
The spectrum of $A$ is also called the spectrum of the signed graph $\Gamma$.
For a vertex subset $X$ of $\Gamma$, the operation that changes the sign of all outgoing edges of $X$, is called switching.
In terms of the matrix $A$, switching multiplies the rows and columns of $A$ corresponding to $X$ by $-1$.
The switching operation gives rise to an equivalence relation, and equivalent signed graphs have the same spectrum
(see \cite[Proposition 3.2]{zaslavsky1}).
If a signed graphs can be switched into an isomorphic copy of another signed graph, the two signed graphs are called
\textit{switching isomorphic}.
Clearly switching isomorphic graphs are cospectral (that is, they have the same spectrum).
A signed graph $\Gamma$ is determined by spectrum whenever every graph cospectral with $\Gamma$ is switching isomorphic with $\Gamma$.
For unsigned graphs it is known that the path $P_n$ is determined by the spectrum of the adjacency matrix,
see \cite[Proposition~1]{dam-hae}.
Among the signed graphs this is in general not true anymore.
In this paper we determine precisely for which $n$ this is still the case, see Theorems~\ref{n-even}, \ref{1nod4}, and
Corollary~\ref{3mod4}.

We refer to \cite{zaslavsky1} and \cite{zaslavsky2} for more information about signed graphs.
For the relevant background on graphs we refer to \cite{brou-haem}, \cite{cvet1}, or \cite{cvet}.
The initial problem was, possibly, first introduced by Acharya in \cite{[1]}.

\section{preliminaries}

A \textit{walk} of length $k$ in a signed graph $\Gamma$ is a
sequence $v_1 e_1 v_2 e_2 \ldots v_k e_k v_{k+1}$ of vertices $v_1, v_2, \ldots, v_{k+1}$ and edges $e_1, e_2, \ldots, e_{k}$
such that $v_i\neq v_{i+1}$ and $e_i=\{v_i,v_{i+1}\}$ for each $i = 1, 2, \ldots, k$.
A walk is said to be \textit{positive} if it contains an even number of negative edges,
otherwise it is called \textit{negative}.
Let $w_{ij}^{+}(k)$ (resp. $w_{ij}^{-}(k)$) denote the number of positive (resp., negative) walks of length $k$ from the
vertex $v_i$ to the vertex $v_j$.
A closed walk is a walk that starts and ends at the same vertex.

In the unsigned case, the $(i, j)$-entry of $A^{k}$ represents the number of walks of length $k$ from $v_i$ to $v_j$.
But in the signed case, powers of $A$ count walks in a signed way.
The $(i, j)$-entry of $A^{k}$ is $w_{ij}^{+}(k) - w_{ij}^{-}(k)$ (\cite[Lemma 3.2]{belardo}, \cite[Theorem II.1]{zaslavsky2}).
For simplicity, we set $W_{k}(\Gamma)=\sum_{i=1}^{n}{(w_{ii}^{+}(k) - w_{ii}^{-}(k))}$.
It is easy to see that if $\Gamma$ and $\Gamma^{'}$ are two cospectral signed graphs,
then $W_k(\Gamma)=W_k(\Gamma^{'})$ for each $k\geqslant 1$.
Moreover, if $\Gamma$ and $\Gamma^{'}$ are two cospectral signed graphs since the sum of the squares of the eigenvalues is twice of the number of edges, we obtain that the order and the size of $\Gamma$ and $\Gamma^{'}$ are the same.

The following lemma can be easily proved by induction.
\begin{lem}\label{walk}
~\\[-27pt]
$$W_4(P_n)=14 + 6(n-4),~for ~n\geqslant 2,$$
$$W_6(P_n)=76 + 20(n-6),~\text{for}~ n\geqslant 3,\ \mbox{and}\ W_6(P_2)=2.$$
\end{lem}
A cycle in a signed graph is called \textit{balanced} if it contains an even number of negative edges,
otherwise it is called \textit{unbalanced}.
A signed graph is balanced if all its circuits are balanced.
It is easily seen that a signed path and a balanced cycle is switching isomorphic with the underlying unsigned path and cycle, respectively.
An unbalanced cycle is switching isomorphic with the underlying cycle with precisely one negative edge.

\begin{lem}\cite[Lemma 4.4]{belardo}.\label{spec}
Let $P_n$ and $C_n$ (resp. $C_{n}^{-}$) be the path and the balanced cycle (resp. unbalanced cycle) on $n$ vertices, respectively.
Then the following hold:
\vspace{-0.1cm}
\begin{align*}
\mathrm{Spec}(C_n)&=\big\{2\mathrm{cos}\dfrac{2i\pi}{n}:~i=0, 1, \ldots, n-1\big\},\\
\mathrm{Spec}(C_{n}^{-})&=\big\{2\mathrm{cos}\dfrac{(2i+1)\pi}{n}:~i=0, 1, \ldots, n-1\big\},\\
\mathrm{Spec}(P_n)&=\big\{2\mathrm{cos}\dfrac{i\pi}{n+1}:~i=1, \ldots, n\big\}.
\end{align*}
\end{lem}
Observe that $C_n$ has largest eigenvalue $2$, and that $C_n^-$ has smallest eigenvalue $-2$ when $n$ is odd,
while all eigenvalues of the path are strictly between $-2$ and $2$.
Moreover, all eigenvalues of the path are simple (have multiplicity $1$), while $C_n$ and $C_n^-$ have (many)
eigenvalues of multiplicity $2$.

Suppose $\Gamma$ is a signed graph of order $n$ with adjacency matrix $A$.
Then we write $\det(\Gamma)$ instead of $\det(A)$.
So $\det(\Gamma)$ equals the product of the eigenvalues of $\Gamma$,
and if $p(x)=a_0+a_1 x +\ldots +a_{n-1}x^{n-1}+x^{n}$ is the characteristic polynomial
of $\Gamma$, then clearly $\det(\Gamma)=a_0=p(0)$.
We define $\det'(\Gamma)=a_1=p'(0)$.
If $\Gamma$ has an eigenvalue $0$, then $\det(\Gamma)=0$, and $\det'(\Gamma)$
is the product of the $n-1$ remaining eigenvalues.

\begin{lem}\label{detpn}
\begin{itemize}
\item[(a)]
If $n$ is even, then $\mathrm{det}(P_n)=(-1)^{\frac{n}{2}}$.
\item[(b)]
If $n$ is odd then $\mathrm{det}^{'}(P_n)=(n+1)/2$.
\end{itemize}
\end{lem}

\begin{proof}
(a)~Clearly $\det(P_2)=-1$, and expanding $\det(P_{n+2})$ with respect to an end vertex of $P_{n+2}$ gives $\det(P_{n+2})=-\det(P_n)$.
\\
(b)~Let $B_n$ be the adjacency matrix of $P_n$.
When $n$ is odd, we can write
\[
B_n=\left[\begin{array}{cc}
O      & N \\
N^\top & O \end{array}\right],\ \mbox{where }
N=\left[\begin{array}{ccccc}
1 & 1      & 0      & \cdots & 0 \\
  & \ddots & \ddots &        &   \\
  &        & \ddots & \ddots &   \\
0 & \cdots & 0      & 1      & 1
\end{array}\right].
\]
The eigenvalues of $B_n^2$ are the eigenvalues of $NN^\top$ together with the eigenvalues of $N^\top N$.
Since $NN^\top$ and $N^\top N$ have the same nonzero eigenvalues it follows that $\det'(B_n)=\det(NN^\top)$.
We easily have that $NN^\top=2I+B_m$, where $m=(n-1)/2$.
Write $d_m=\det(2I+B_m)$, then $d_1=2$, $d_2=3$ and $d_{m+2}=2d_{m+1}-d_m$, so $d_m=m+1=(n+1)/2$.
\end{proof}
\begin{lem}\label{ddet}
Let $B$ be a symmetric matrix of order $n$ with two equal rows (and columns),
and let $B'$ be the matrix of order $n-1$ obtained from $B$ by deleting one repeated row and column.
Then $\det(B)=0$, and $\det'(B)=2\det(B')$.
\end{lem}

\begin{proof}
Clearly $B$ is singular, so $\det(B)=0$.
Without loss of generality we assume that the first two rows and columns of $B$ are equal.
Consider the following orthogonal matrices
$Q_2=\frac{1}{\sqrt{2}}\left[\begin{array}{cc}1&1\\-1&1\end{array}\right]$, and $Q=\left[\begin{array}{cc}Q_2&O\\O&I_{n-2}\end{array}\right]$.
Then
$Q^\top BQ= \left[\begin{array}{cc}0&\underline{0}^\top\\ \underline{0}&B''\end{array}\right]$,
where $B''$ is obtained from $B'$ by multiplying the first row and column by $\sqrt{2}$.
On the other hand, $B$ and $Q^\top BQ$ are cospectral, therefore Spec$(B'')=\mbox{Spec}(B)\setminus\{0\}$.
So $\det'(B)=\det(B'')=2\det(B')$.
\end{proof}

\section{Signed graphs cospectral with the path}

In the remaining of the paper we assume that $\Gamma$ is a signed graph cospectral but not switching isomorphic with the path $P_n$.
We know that $\Gamma$ has $n$ vertices and $n-1$ edges.
Since $\Gamma$ is not a signed path, $\Gamma$ has at least two components.
In this section we obtain conditions for the components of $\Gamma$. \\
Graph $D_m$ in Fig.~\ref{nonsimple}, is the union of $K_{1, 3}$ and $P_{m-4}$, where an end vertex of $P_{m-4}$ is joined to a vertex of degree one in $K_{1, 3}$.
\begin{obs}\label{obs2}
\begin{enumerate}
\item
By the interlacing theorem and Lemma \ref{spec}, $\Gamma$ contains no odd cycle,
no balanced even cycle, and no star $K_{1, 4}$ as an induced subgraph, (note that the biggest adjacency eigenvalue of
$K_{1, 4}$ is $2$).
Hence, all cycles in $\Gamma$ are unbalanced of even order, and the maximum degree of $\Gamma$ is at most $3$.
\item
We checked (by computer) that a signed graph for which the underlying unsigned graph is one of the graphs given in Fig. \ref{greater2}
has largest eigenvalue at least $2$.
Therefore, no graph in Fig. \ref{greater2} has an induced subgraph of $\Gamma$.
Also each graph of Fig.~\ref{nonsimple} has at least one eigenvalue of multiplicity at least $2$.
Therefore none of these can be a component of $\Gamma$.
Note that Graph~(g) in Fig.~\ref{nonsimple}, has an eigenvalue of multiplicity $3$, so by the interlacing theorem, each graph on $8$ vertices having Graph~(g) as an induced subgraph has at least one non-simple eigenvalue, and therefore cannot be a component of $\Gamma$.

\item
Let $M$ be Graph~(e) of Fig.~\ref{nonsimple}.
Then $M$ is not an induced subgraph of $\Gamma$.
Indeed, $M$ is not a component of $\Gamma$, and every graph on $9$ vertices with maximum
degree $3$ that contains $M$ as an induced subgraph contains an odd cycle, or Graph~(a) from Fig.~\ref{greater2}.

\item
A $\Theta$-graph is a union of three internally disjoint paths $P_{p}, P_{q}, P_{r}$ with
common end vertices, where $p, q, r \geqslant 2$ and at most one of them equals $2$.
If $p, q, r \geqslant 3$ we call the $\Theta$-graph \textit{proper}.
A proper signed $\Theta$-graph has at least one balanced cycle.
Then using the interlacing theorem for this induced balanced cycle,
we conclude that a $\Gamma$ has no proper $\Theta$-graph as an induced subgraph.

\begin{figure}[!htb]
\minipage{0.90\textwidth}
\includegraphics[width=\linewidth]{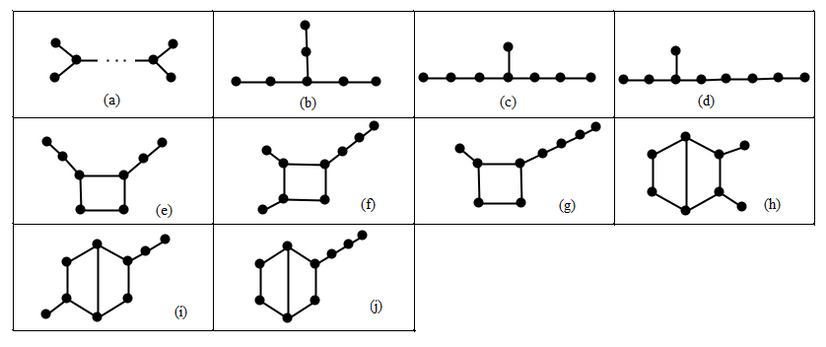}
\caption{Graphs with largest eigenvalue at least $2$}\label{greater2}
\endminipage
\end{figure}

\begin{figure}[!htb]
\minipage{0.88\textwidth}
\includegraphics[width=\linewidth]{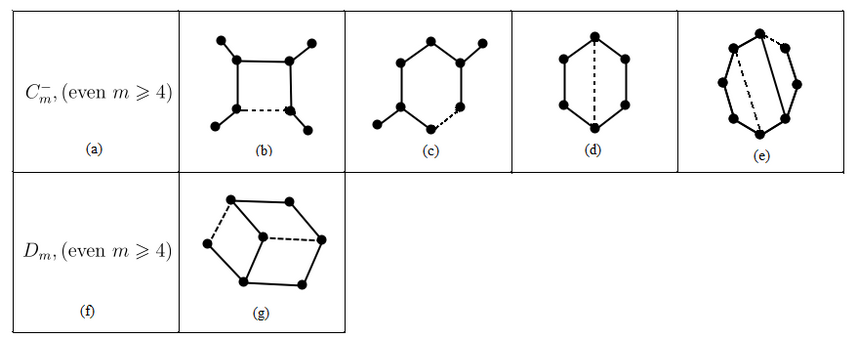}
\caption{Graphs with some non-simple eigenvalues (dashed edges are negative)}\label{nonsimple}
\endminipage
\end{figure}

\item
An unbalanced even cycle has eigenvalues of multiplicity $2$, and therefore cannot occur as a component of $\Gamma$.
Furthermore, we claim that $\Gamma$ contains no induced even cycle of order more than $6$.
Indeed, let $C_r^-$ be an unbalanced induced cycle for $r\geqslant 8$.
Since $C_r^-$ is not a component, there exist a vertex $v$ out of $C_r^-$, which is adjacent to one, two or three vertices of $C_r^-$.
If $v$ is adjacent to just one vertex of $C_r^-$, then we have Graph $(c)$ given in Fig. \ref{greater2} as an induced subgraph.
If $v$ is adjacent to two vertices of $C_r^-$, then graph $\langle V(C_r) \cup \{v\}\rangle$ is a proper $\Theta$-graph
(recall that $\Gamma$ has no odd cycle).
If $v$ is adjacent to three vertices of $C_r^-$, then it is easy to check that the graph $\langle V(C_r) \cup \{v\}\rangle$ has
one of the graphs $(a)$ or $(b)$ given in Fig. \ref{greater2} as an induced subgraph.
In all three case we have a contradiction, and since no vertex of $\Gamma$ has degree more than three, the claim is proved.
\end{enumerate}
\end{obs}

In \cite{mckee} the authors have classified signed graphs having all their
eigenvalues in the interval $[-2, 2]$.
Now, based on \cite[Theorem 4]{mckee} and Observation \ref{obs2} we have the following results.

\begin{lem}\label{C6}
If $H$ is a component of $\Gamma$ containing an induced $6$-cycle, then $H$ is one of the graphs presented in Fig.~\ref{6-cycle}
\end{lem}
\begin{figure}[!htb]
\minipage{0.45\textwidth}
\includegraphics[width=\linewidth]{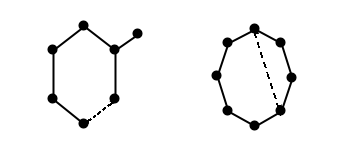}
 \caption{}\label{6-cycle}
\endminipage
\end{figure}
\begin{proof}
We know that the $6$-cycle is unbalanced.
Also, the unbalanced $6$-cycle $C_6^-$ has eigenvalues of multiplicity $2$, so $H\neq C_6^-$.
By \cite[Theorem 4]{mckee} and Fig.~\ref{nonsimple}, there are
only two types for component $H$, see Fig.~\ref{6-cycle}.
\end{proof}

\begin{thm}\label{main}
The only graphs that can occur as a component of $\Gamma$ are listed in Fig.~\ref{1}.
\end{thm}
\begin{proof}
Let $H$ be a component of $\Gamma$.
If $H$ is a tree, then since all eigenvalues are strictly less than $2$,
$H$ is one of the trees in Fig.~\ref{1} (see \cite[Theorem 3.1.3]{brou-haem}). Note that $K_{1, 3}$ is not among the graphs given in Fig.~\ref{1}, because $K_{1, 3}$ has a non-simple eigenvalue.
Now, suppose that $H$ has a cycle.
By Observation \ref{obs2}, $H$ has no induced unbalanced cycle of order more than $6$.
Hence, every induced cycle of $H$ has order $4$ or $6$.
If $H$ has an induced $6$-cycle, then $H$ is Graph (f) or (j) in Fig.~\ref{1}, by Lemma \ref{C6}.

Now, assume that $H$ has an induced unbalanced $4$-cycle but no induced $6$-cycle.
If $H=C_4^-$, then $\Gamma$ has non-simple eigenvalues.
By \cite[Theorem 4]{mckee}, Observation~\ref{obs2} and Figs.~\ref{greater2} and \ref{nonsimple}, $H$ is one of the graphs given in Fig.~\ref{1}.
\end{proof}

\begin{figure}[!htb]
\minipage{0.98\textwidth}
\includegraphics[width=\linewidth]{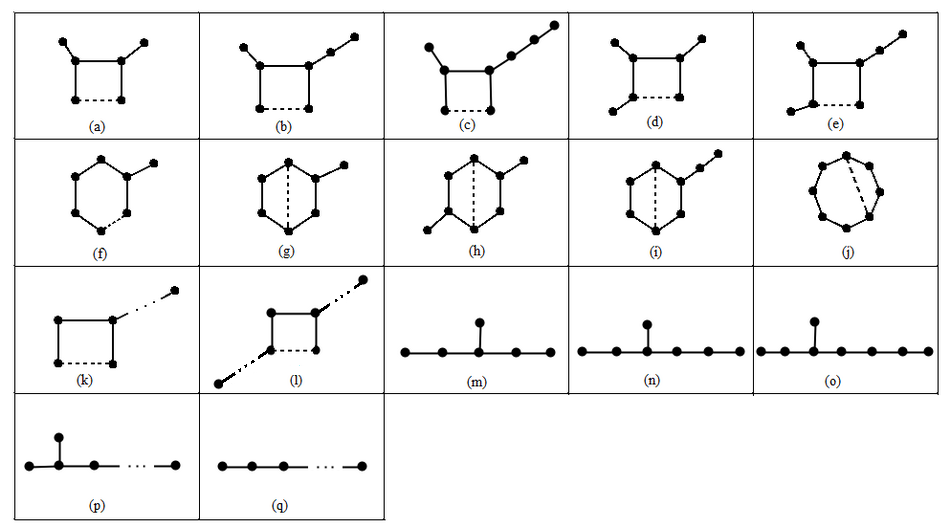}
 \caption{All possible components for $\Gamma$}\label{1}
\endminipage
\end{figure}

\begin{thm}
If $\Gamma$ is cospectral, but not switching isomorphic to $P_n$, then $\Gamma$ contains an unbalanced 4-cycle as induced subgraph.
\end{thm}

\begin{proof}
Suppose $\Gamma$ does not contain an $C_4^-$.
Then, since $\Gamma$ contains an induced unbalanced cycle $C_r^-$ with $r\leqslant 6$,
Theorem~\ref{main} implies that Graph (f) of Fig.~\ref{1} is a component of $\Gamma$.
Also there are not two or more components isomorphic to Graph~(f), since then $\Gamma$ would have eigenvalues of multiplicity
at least $2$.
So we can conclude that $\Gamma$ has just one non-tree component, which is Graph~(f),
and there is just one more component isomorphic to (m), (n), (o), (p), or (q) of Fig.~\ref{1} because the size of $\Gamma$ should be equal to the order of $\Gamma$ minus $1$.
Moreover, the reader can find the spectrum of Graphs ~(m), (n), (o), (p), and (q) in Fig.~\ref{1} on \cite[Theorem 3.1.3]{brou-haem}. By verification it follows that none of these possibilities has the spectrum of $P_n$.
\end{proof}

Note that only four cases in Fig.~\ref{1} represent an infinite family.
Graph~(q) of order $m$ is the path $P_m$, and Graph~(p) of order $m$ is known as $D_m$.
Graph~(k) and (l) will be denoted by $H_t$ and $H_{t}^{t+m}$, respectively.
More precisely, $H_t$ is the union of $C_4^-$ and $P_t$, where an end vertex of $P_t$ is joined to a vertex of $C_4^-$,
and $H_t^{t+m}$ is the union of $C_4^-$, $P_t$ and $P_{m+t}$ where an end vertex of $P_t$ is joined to one vertex of $C_4^-$,
and an end vertex of $P_{m+t}$ is joined to the opposite vertex of $C_4^-$.
\begin{lem}\label{lem11}
For integers $t\geqslant 1$, $k\geqslant 0$ and $m\geq 1$, $\mathrm{det}(H_t)$, $\mathrm{det}(H_{t}^{t+m})$,
$\mathrm{det}(D_{m})$, $\mathrm{det'}(H_t)$, $\mathrm{det'}(H_{t}^{t+m})$,
and $\mathrm{det'}(D_{m})$ are even.
\end{lem}
\begin{proof}
Let $B$ be the adjacency matrix of the underlying unsigned graphs of $H_t$ or $H_{t}^{t+m}$.
Then $B$ contains two repeated rows (and columns), so by Lemma~\ref{ddet}
$\det(B)=0$ and $\det'(B)=2\det(B')$, so $\det'(B)$ is even.
On the other hand, the signed and the unsigned graph have equal adjacency matrices modulo $2$.
\end{proof}

\begin{thm}\label{eigH}
The eigenvalues of $H_{t}^{t+m}$ ($t\geq 1$, $m\geq 0$) are as follows:
The first type of eigenvalues are
$$2\mathrm{cos}\dfrac{(2i-1)\pi}{2k},~\text{for}~k=t+m+2,~i=1, \ldots , k.$$
The second type of eigenvalues are
$$2\mathrm{cos}\dfrac{(2i-1)\pi}{2t+4},~\text{for}~i=1, \ldots , t+2.$$
\end{thm}
\begin{proof}
Consider a labeling of $H_{t}^{t+m}$, $t$ and $m$ are even or odd, as presented in Figs.~\ref{H2,1} and \ref{H2,2}. Note that the sets $\{v_1, v_2, \ldots\}$ and $\{u_1, u_2, \ldots\}$ in Figs.~\ref{H2,1} and \ref{H2,2}, are two appropriate partitions of signed bipartite graph $H_{t}^{t+m}$.
Hence, we can write the adjacency matrix $A$ of $H_{t}^{t+m}$ as follows:

\begin{equation*}
A =
\begin{bmatrix}
O & N \\
N^T & O\\
\end{bmatrix}.
\end{equation*}
Then it is seen that
\begin{equation*}
A^2 =
\begin{bmatrix}
NN^T & O \\
O & N^TN \\
\end{bmatrix}.
\end{equation*}
We can write $NN^T$ as the following matrix
\begin{equation*}
NN^T=
\begin{bmatrix}
K & O \\
O & L \\
\end{bmatrix},
\end{equation*}
where $K$ and $L$ are tridiagonal matrices with all-ones on the upper and lower diagonal,
and $\left[ 3, 2, 2, \ldots , 2, 1\right ]$ or $\left[ 3, 2, 2, \ldots , 2 \right ]$ on the diagonal.

\begin{figure}[!htb]
\minipage{0.80\textwidth}
\includegraphics[width=\linewidth]{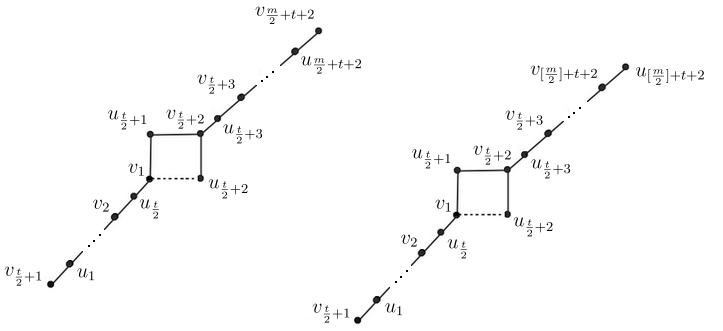}
 \caption{Labeling of $H_{t}^{t+m}$ for even $t$}\label{H2,1}
\endminipage
\end{figure}
\newpage
\begin{figure}[!htb]
\minipage{0.80\textwidth}
\includegraphics[width=\linewidth]{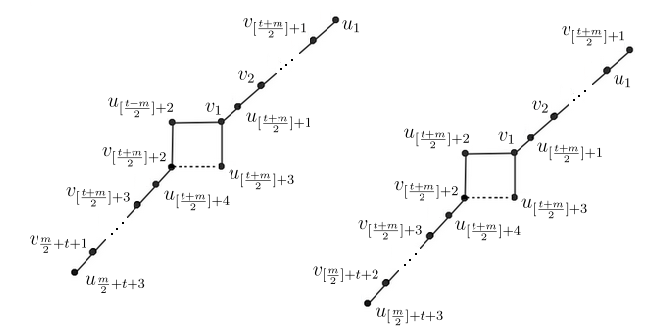}
 \caption{Labeling of $H_{t}^{t+m}$ for odd $t$}\label{H2,2}
\endminipage
\end{figure}

Assume that $K$ and $L$ are square matrices of size $s$ and $r$, respectively.
If $t$ is even, then $s=\frac{t}{2}+1$ and $r=[\frac{t+m}{2}]+1$.
Otherwise $s=[\frac{t+m}{2}]+1$ and $r=\lceil \frac{t+1}{2}\rceil$.
Moreover, by \cite[Theorems 2, 3]{eigtri}, we can obtain the eigenvalues of $K$ and $L$ using the following equalities, respectively.
$$\lambda_j = 2 + 2 \mathrm{cos}\dfrac{(2j-1)\pi}{2s}, ~j=1, 2, \ldots, s,$$
$$\lambda_i= 2 + 2 \mathrm{cos}\dfrac{(2i-1)\pi}{2r+1}, ~i=1, 2, \ldots, r.$$
Now, using a simple trigonometric relation the assertion is proved.
\end{proof}

\section{Paths of even order}

Suppose $n$ is even.
By Lemma \ref{detpn} $\det(\Gamma)=\det(P_n)=(-1)^{\frac{n}{2}}$.
Therefore each component of $\Gamma$ has determinant $+1$ or $-1$,
hence Graphs (a), (c), (e), (h), (i), (j), (m), (o), and (q) ($P_k$ with $k$ even) given in
Fig.~\ref{1} are the only possible components of $\Gamma$.

\begin{lem}\label{cos.spec}
The spectrum of Graphs (c), (e) and (i) in Fig.~\ref{1} are as follows:
\begin{align*}
\mathrm{Spec}(c)&=\big\{2\mathrm{cos}\dfrac{k\pi}{24}:~k=1, 5, 7, 11, 13, 17, 19, 23\big\},\\
\mathrm{Spec}(e)&=\big\{2\mathrm{cos}\dfrac{k\pi}{20}:~k=1, 3, 7, 9, 11, 13, 17, 19\big\},\\
\mathrm{Spec}(i)&=\big\{2\mathrm{cos}\dfrac{k\pi}{18}:~k=1, 3, 5, 7, 11, 13, 15, 17\big\}.
\end{align*}
\end{lem}

By \cite[Theorem 3.1.3]{brou-haem} and Lemma \ref{cos.spec}, Graphs (c), (e), (i) and (m) cannot be a component of a signed graph cospectral with $P_n$ for any even $n$.
Therefore the only graphs which can occur as a component of $\Gamma$ for even $n$ are the
graphs presented in Fig.~\ref{components}.
\begin{figure}[!htb]
\minipage{0.85\textwidth}
\includegraphics[width=\linewidth]{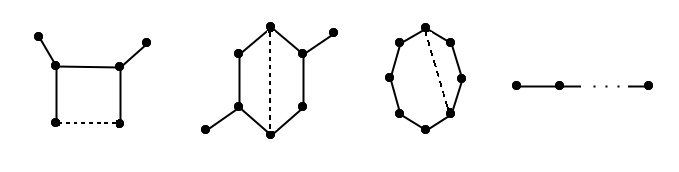}
 \caption{All possible components of $\Gamma$}\label{components}
\endminipage
\end{figure}

We note that the second and the third graph in Fig.~\ref{components} are cospectral.
Therefore, at most one of them can be a component of $\Gamma$.\\
\begin{lem}\label{lem.1}
If $n$ is even and $\Gamma$ has two connected components then $n=8$.
Moreover, $\Gamma$ is switching isomorphic with the disjoint union of $P_2$ and Graph~(a)
from Fig.~\ref{1}.
\end{lem}
\begin{proof}
Based on the possible components for $\Gamma$ in Fig. \ref{components},
we have only one type of $\Gamma$ with two components, being Graph~(a) and $P_{n-6}$.
By considering the values of $W_{4}(\Gamma)$ and $W_{6}(\Gamma)$ and using Lemma \ref{walk}, we have
$W_4(\Gamma) = W_4(P_n)$ for each $n$, but $W_6(\Gamma) \neq W_6(P_n)$ for $n\neq 8$.
If $n=8$ it is easily verified that $\Gamma$ and $P_8$ are cospectral.
\end{proof}

\begin{lem}\label{lem.2}
If $n$ is even and $\Gamma$ has three connected components then $n=14$.
Moreover, if $n=14$, then $\Gamma$ is switching isomorphic with the disjoint union of either $P_2$, $P_4$ and Graph~(h),
or $P_2$, $P_4$ and Graph~(j) in Fig. \ref{1}.
\end{lem}
\begin{proof}
After considering all cases of the components of $\Gamma$ in Fig.~\ref{components},
we obtain two types of $\Gamma$ with three components given in Fig.~\ref{3-component}.
These two possible types of $\Gamma$ are similar because the spectrum of the first components are the same.
Hence, it is sufficient to verify one of these two cases for $\Gamma$.
We note that when $\Gamma$ contains two paths, then the orders of the paths are different because
otherwise the multiplicity of some of the eigenvalues will be at least two.
We have $W_4(\Gamma)=W_4(P_n)$, but  $W_6(\Gamma)\neq W_6(P_n)$, unless $n=14$.
By an easy inspection, we conclude that if $n=14$ and the path components have orders $2$ and $4$,
then $\mathrm{Spec}(\Gamma)=\mathrm{Spec}(P_{14})$.
\end{proof}
\vspace{-0.2 cm}
\begin{figure}[!htb]
\minipage{0.85\textwidth}
\includegraphics[width=\linewidth]{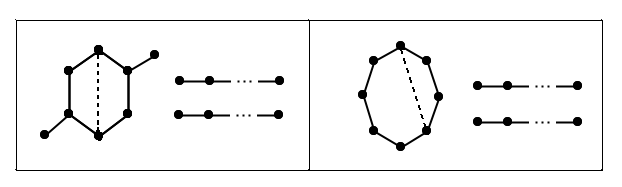}
 \caption{All types of $\Gamma$ with three components and $n$ even}\label{3-component}
\endminipage
\end{figure}
\vspace{0.2 cm}

\begin{lem}\label{lem.3}
If $n$ is even, then $\Gamma$ has at most three components.
\end{lem}
\begin{proof}
Assume that $\Gamma$ has more than three components.
Using Fig.~\ref{components} we see that there are only the two types for $\Gamma$ shown in Fig. \ref{4-component}.
Similar to the proof of Lemmas \ref{lem.1} and \ref{lem.2}, it is sufficient to determine $W_6$ for $\Gamma$ and $P_n$.
In each case we achieve a contradiction.
\end{proof}
\begin{figure}[!htb]
\minipage{0.90\textwidth}
\includegraphics[width=\linewidth]{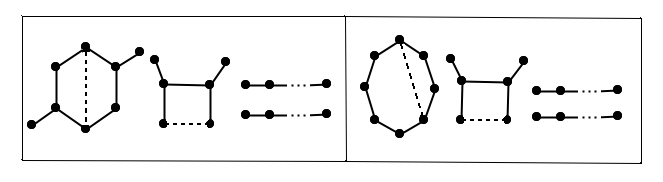}
 \caption{All types of $\Gamma$ with four components and $n$ even}\label{4-component}
\endminipage
\end{figure}

\begin{thm}\label{n-even}
Suppose $n$ is even.
Then $P_n$ is determined by the spectrum if and only if $n\neq 8, 14$.
\end{thm}

\begin{proof}
It follows from Lemmas \ref{lem.1}, \ref{lem.2} and \ref{lem.3}.
\end{proof}

\section{Paths of the odd order}

\begin{thm}\label{1nod4}
Suppose $n\equiv 1~(\mathrm{mod}~4)$.
Then $P_n$ is determined by the spectrum if and only if $n\not\in\{13, 17, 29\}$.
\end{thm}
\begin{proof}
Since $n$ is odd, $\det(\Gamma)=0$, and exactly one component $H$ of $\Gamma$ has an eigenvalue $0$.
The product of all other eigenvalues of $\Gamma$ equals $\det'(\Gamma)=(n+1)/2$ by Lemma \ref{detpn}.
Since $(n+1)/2$ is odd, $\det'(H)$ is odd, and every component different from $H$ has an odd determinant.
Hence, by Lemma \ref{lem11}, the possible candidates do not include Graphs~(k), (l) and (p) in given Fig.~\ref{1}.
So, there is only a small list of possible components of $\Gamma$.
Clearly $\lambda_1(P_n)$ is equal to $\lambda_1(H)$ for one of the components of $\Gamma$.
Since $\lambda_1(P_k)<\lambda_1(P_n)$ when $k<n$, $H\neq P_k$,
and the largest eigenvalue of each of the other possible components is at most $\lambda_1(P_{29})$.
Therefore $P_n$ is determined by the spectrum when $n\geq 33$.

For $n=5, 9$, it is easy to check that $P_n$ is determined by the spectrum.
If $n=21, 25$, then $\det'(P_n)=11, 13$ respectively.
But none of the components $H$ in Fig.~\ref{1} (except $P_{21}$ and $P_{25}$) has $\det(H)$, or $\det'(H)$
equal to $11$ or $13$.
Hence, $P_{21}$ and $P_{25}$ are determined by their spectrums.
Furthermore, we give graphs cospectral with $P_{13}$, $P_{17}$ and $P_{29}$ in Fig.~\ref{counterexamples}.
\end{proof}
\begin{thm}\label{odd case2}
Let $n=4k+3$ for some integer $k\geq 1$.
Then there exists a graph $\Gamma$ which is cospectral but not switching isomorphic with $P_n$.
\end{thm}
\begin{proof}
Consider graph $\Gamma$ with two components $H_2$ and $P_1$.
It is easy to check that $\Gamma$ is cospectral with $P_7$.
For other cases, we show that a signed graph with two components $H_{k-1}^{2k}$ and $P_k$ is a cospectral mate of $P_{4k+3}$.
\begin{align*}
&\mathrm{Spec}(P_{4k+3})=\{2\mathrm{cos}\dfrac{i\pi}{4k+4}, ~i=1, 2, \ldots, 4k+3\}\\
& =\{2\mathrm{cos}\dfrac{i\pi}{4(k+1)}, ~i=1, 3, \ldots, 4k+3\}\cup \{2\mathrm{cos}\dfrac{j\pi}{4(k+1)}, ~j=2, 4, \ldots, 4k+2\}\\
& =\{2\mathrm{cos}\dfrac{i\pi}{4(k+1)}, ~i=1, 3, \ldots, 4k+3\}\cup \{2\mathrm{cos}\dfrac{j\pi}{2(k+1)}, ~j=1, 2, \ldots, 2k+1\}\\
&=\mathrm{Spec}(H_{k-1}^{2k})\cup \mathrm{Spec}(P_k).
\end{align*}
~\\[-35pt]
\end{proof}
~\\
In Fig.~\ref{counterexamples} ($E_6$ is Graph (m) of Fig.~\ref{1}, and $E_8$ is Graph (o) of Fig.~\ref{1}),
we give signed graphs cospectral with $P_{11}, P_{15}$ and $P_{23}$.
It shows that the presented graphs in Theorem~\ref{odd case2} are in general not unique.

Obviously $P_3$ is determined by its spectrum, so we have the following conclusion.

\begin{cor}\label{3mod4}
Suppose $n\equiv 3 \mod 4$.
Then $P_n$ is determined by its spectrum if and only if $n=3$.
\end{cor}

\begin{figure}[!htb]
\minipage{0.88\textwidth}
\includegraphics[width=\linewidth]{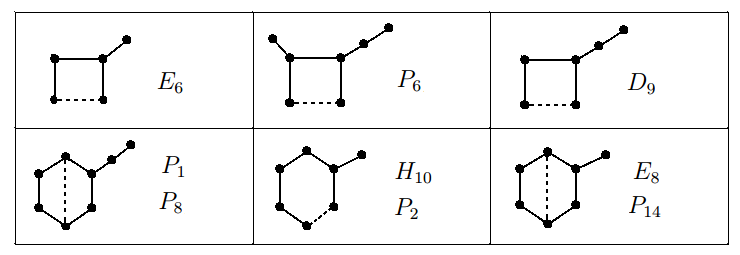}
\caption{Cospectral mates of $P_{11}, P_{13}, P_{15}, P_{17}, P_{23}$, $P_{29}$}\label{counterexamples}
\endminipage
\end{figure}


\section*{Acknowledgment}

The research of the first author was partly funded by Iran National Science Foundation (INSF) under the contract No. 96004167.
\\
We would like to thank an anonymous referee for comments and suggestions.



\end{document}